\documentclass[11pt,reqno]{amsart}
\usepackage{amsmath,amssymb,amsthm, hyperref, amscd}
\usepackage[T1]{fontenc}
\usepackage{textcomp}
\usepackage[scale=0.95,ttdefault]{AnonymousPro}
\hypersetup{
    colorlinks=true,
    linkcolor=blue,
    urlcolor=cyan,
    }
\usepackage[noadjust]{cite}
\title{PERMANENCE PROPERTIES OF SPLINTERS VIA ULTRAPOWER}

\newcommand{\fa}{\mathfrak{a}}

\newcommand{\fm}{\mathfrak{m}}
\newcommand{\fM}{\mathfrak{M}}

\newcommand{\fp}{\mathfrak{p}}
\newcommand{\fq}{\mathfrak{q}}
\newcommand{\ft}{\mathfrak{t}}

\newcommand{\Spec}{\operatorname{Spec}}
\newcommand{\Spl}{\operatorname{Spl}}
\newcommand{\cO}{\mathcal{O}}
\newcommand{\cU}{\mathcal{U}}
\newcommand{\bA}{\mathbf{A}}
\newcommand{\bF}{\mathbf{F}}
\newcommand{\bQ}{\mathbf{Q}}

\newcommand{\Tag}[1]{\href{https://stacks.math.columbia.edu/tag/#1}{\texttt{#1}}}
\newcommand{\citestacks}[1]{\cite[Tag \Tag{#1}]{Stacks}}
\newcommand{\citetwostacks}[2]{\cite[Tags \Tag{#1} and \Tag{#2}]{Stacks}}

\author{Shiji Lyu}

\newtheorem{Thm}{Theorem}[section]

\newtheorem{Lem}[Thm]{Lemma}
\newtheorem{Cor}[Thm]{Corollary}
\newtheorem{Prop}[Thm]{Proposition}

\theoremstyle{definition}
\newtheorem{Def}[Thm]{Definition}
\begin{document}
\begin{abstract}
   We show that the splinter property ascends along regular residue field extensions, and along arbitrary regular maps in equal characteristic.
   We also study the splinter property of non-Noetherian rings, especially those related to ultrapowers, to the extent necessary for our main results.
\end{abstract}
\maketitle
\section{Introduction}
\

A Noetherian ring is called a splinter if it is reduced and is a direct summand of every finite extension, see Definition \ref{def:splinter} and Lemma \ref{lem:SplinterNormal} below.
Definition being simple, the properties of this notion are usually difficult to prove or understand.
The direct summand theorem,
that a regular Noetherian ring is a splinter, is only recently fully proved, see the celebrated works {\citeleft\citen{Hoc73}\citepunct\citen{And18}\citepunct\citen{Bh18}\citeright}.

The natural question to ask next is that if the splinter property ascends along a regular map.
Built on recent advances \citeleft\citen{DTetale}\citepunct\citen{DTopen}\citeright,
we show that this is true in equal characteristic (Theorem \ref{thm:RegCharp});
In mixed characteristic, we show that the splinter property ascends along a ``regular residue field extension'' (Theorem \ref{thm:ScalarFoverR}).

\begin{Thm}\label{thm:RegCharp}
Let $S\to A$ be a regular homomorphism of Noetherian rings.
Assume that $S$ contains a field. If $S$ is a splinter, so is $A$.
\end{Thm}

Taking $S=\bQ$ or $\bF_p$ for a prime number $p$, we obtain the direct summand theorem in equal characteristic \cite[Theorem 2]{Hoc73}.
\begin{Cor}\label{cor:DirectSummand}
Let $A$ be a regular Noetherian ring that contains a field.
Then $A$ is a splinter.
\end{Cor}

The following theorem is valid in arbitrary characteristic and
has no excellence hypothesis.
In view of
\cite[Lemma 2.1.3]{DTetale},
this theorem generalizes {\cite[Theorem A]{DTetale}}, a theorem used in the proof.
We also note that taking $S'$ to be the completion of $S$ we obtain an alternative proof of {\cite[Theorem C]{DTetale}}.

\begin{Thm}\label{thm:ScalarFoverR}
Let $(S,\fm)\to (S',\fm')$ be a regular homomorphism of Noetherian local rings with $\fm'=\fm S'$.
If $S$ is a splinter, so is $S'.$
\end{Thm}

The following corollary is the special case where $S$ is a G-ring, in view of a theorem of Andr\'e (\cite{Andre} or \citestacks{07PM}).
We will explain this in detail after the proof of Theorem \ref{thm:ScalarFoverR} for the reader's convenience.
Readers not familiar with the notion of a G-ring can refer to \citetwostacks{07GG}{07QS}.
We mention that all excellent rings are G-rings by definition, and that for a Noetherian local ring being a G-ring is the same as being quasi-excellent, see \cite[(34.A)]{Matsumura}.

\begin{Cor}\label{cor:ScalarG}
Let $(S,\fm,k)\to (S',\fm',k')$ be a flat homomorphism of Noetherian local rings.
Assume that $\fm'=\fm S'$, that $k'/k$ is separable,
and that $S$ is a G-ring.
If $S$ is a splinter, so is $S'$.
\end{Cor}

Here, a separable extension of fields does not need to be algebraic, finitely generated, or separably generated; see \citetwostacks{030O}{0322}.\\

We now discuss the strategy to prove the theorems above. See Section \ref{sec:proofs} for explicit arguments.
For Theorem \ref{thm:ScalarFoverR},
we dominate the ring $S'$ with an ultrapower of $S$ which we introduce and study in an ad-hoc fashion in Section \ref{sec:Ultraproduct}.
For the splinter property to be well-defined and preserved, we need to study non-Noetherian splinters in Section \ref{sec:splinters}, which is partially done in \cite{ADVal}.
In this section we also establish an openness result (Theorem \ref{thm:OPENsmoverFpure}) parallel to the main result of \cite{DTopen}. 
For Theorem \ref{thm:RegCharp}, we can reduce to the case of a polynomial algebra as observed in \cite{DTetale},
whose symmetry we can exploit with Theorem \ref{thm:ScalarFoverR} and
the openness result established in Section \ref{sec:splinters}.
\\

We close our introduction by announcing that using similar methods (but with more technical preparations and struggles) the author is able to establish Theorem \ref{thm:ScalarFoverR}
for birational derived splinters and openness and regular ascent for $F$-pure birational derived splinters.
Combined with \cite{Kovacs}, the same results hold if we replace birational derived splinters by rational singularities,
once we make a good definition.
This is work in preparation.
\\

\textsc{Acknowledgements}. We thank Takumi Murayama for helpful discussions and comments on a draft.
We thank Rankeya Datta and Linquan Ma for various communications, including reading a draft and give comments, as well as inspirations
that helped significantly simplify the current exposition.
We also thank Kevin Tucker for helpful discussions.
\section{Splinters}\label{sec:splinters}

\begin{Def}[see {\cite[Definition 5.2.2]{ADVal}}]
\label{def:splinter}
A ring
$S$ is a \emph{splinter} if every finite and finitely presented ring map $S\to T$
that induces a surjective map of spectra
splits as a map of $S$-modules.
Equivalently, $S$ is a splinter if
every finite ring map $S\to T$ that
induces a surjective map of spectra is pure, see  \cite[Lemma 5.2.1]{ADVal}.
Here, and everywhere in this article, ``pure'' means ``universally injective'' as in
\citestacks{058I}.
\end{Def}

The next two lemmas are included in \cite[Proposition 5.2.5]{ADVal}.

\begin{Lem}
\label{lem:FlatDescentSplinters}
Let $S\to S'$ be a pure (for example faithfully flat) ring map.
If $S'$ is a splinter, so is $S$.
\end{Lem}

\begin{Lem}
\label{lem:DlimSplinters}
A direct limit of splinters is a splinter.
\end{Lem}

\begin{Lem}\label{lem:SplinterProduct}
Let $X$ be an index set, $S_x\ (x\in X)$ a family of splinters.
Then the direct product $S=\prod_x S_x$ is a splinter.
\end{Lem}
\begin{proof}
Let $\varphi:S\to T$ be a
finite and finitely presented ring map that induces a surjective map of spectra.
We need to show $\varphi$ splits.

Let $p_x:S\to S_x$ be the projection, and $\varphi_x:S_x\to T_x$ the base change of $\varphi$ along $p_x$.
Then $\varphi_x$ is
a finite and finitely presented ring map that induces a surjective map of spectra.
Since $S_x$ is a splinter, $\varphi_x$ splits as a map of $S_x$-modules.
Taking product, we see that the map $\psi:S=\prod_x S_x\to\prod_x T_x$ splits as a map of $S$-modules and canonically factors through $\varphi$, thus $\varphi$ splits, as desired. (Actually $\psi=\varphi$, see \citestacks{059K}.)
\end{proof}

The next few lemmas are probably well-known in the Noetherian case, but not necessarily in this generality. We give proofs.

\begin{Lem}\label{lem:FFPSLocalize}
Let $A$ be a ring, $\fp\in\Spec(A)$.
Let $A_\fp\to B_1$ be a finite and finitely presented ring map  that induces a surjective map of spectra.
Then there exists
a finite and finitely presented ring map $A\to B$ that induces a surjective map of spectra such that $B_\fp=B_1$.
\end{Lem}
\begin{proof}
Let $B'$ be the integral closure of $A$ in $B_1$.
Then $B_1=B'_\fp$ by \citestacks{0307}.
Since $B_1$ is finite over $A_\fp$, ther exists a finite $A$-subalgebra $B''$ of $B'$ such that $B''_\fp=B_1.$
Write $B''=\mathrm{colim}_iB''_i$ where each $B''_i$ is finite and finitely presented over $A$ and the transition maps are surjective,
see \citestacks{09YY}.
In particular $B''_i\to B''$ is surjective for each $i$.
For large enough $i$, $(B''_i)_\fp\to B''_\fp=B_1$ is split injective as a map of $A_\fp$-algebras,
see \citestacks{00QO}, thus an isomorphism.

We have now obtained a finite and finitely presented $A$-algebra $B_0=B''_i$
such that $(B_0)_\fp=B_1$.
Since $\Spec(B_1)\to\Spec(A_\fp)$
is surjective, there exists $f\not\in \fp$ such that
$\Spec((B_0)_f)\to\Spec(A_f)$
is surjective,
see \citestacks{07RR}.
Set $B=B_0\times A/fA$.
Then by our choice, $B_\fp=B_1$,
$B$ is a finite and finitely presented $A$-algebra,
and $\Spec(B)\to \Spec(A)$ is surjective, as desired.
\end{proof}

\begin{Cor}[cf. {\cite[Lemma 2.1.3]{DTetale}}]
\label{cor:SplinterLocalize}
The followings are equivalent for an arbitrary ring $S$.

(1) $S$ is a splinter.

(2) For every prime ideal $\fp$ of $S$, $S_\fp$ is a splinter.

(3) For every maximal ideal $\fm$ of $S$, $S_\fm$ is a splinter.
\end{Cor}
\begin{proof}
Lemma \ref{lem:FFPSLocalize} gives that (1) implies (2);
(2) trivially implies (3); and (3) implies (1) because purity can be checked locally, see \citestacks{05CL}.
\end{proof}

\begin{Lem}[cf. {\cite[Lemma 2.1.1]{DTetale}}]
\label{lem:SplinterNormal}
Let $S$ be a splinter.
Then $S$ is reduced and integrally closed in its total fraction ring.
If $S$ has finitely many minimal primes, then $S$ is a finite product of normal domains.
\end{Lem}
\begin{proof}
Let $a\in S$ be nilpotent.
Then $S\to S/aS$ is a finite and finitely presented ring map that induces a surjective map on spectra.
Therefore it is split injective and $a=0$.
Thus $S$ is reduced.

Let $T$ be a finite $S$-subalgebra of the total fraction ring $K$ of $S$.
Then $S\to T$ is injective and finite, so $\Spec(T)\to\Spec(S)$ is surjective.
Thus $S\to T$ is
pure (Definition \ref{def:splinter}), so $aT\cap S=aS$ for all $a\in S$.
Since $K$ is a localization of $S$ we see $S=T$.

The last statement follows from \citestacks{030C}.
\end{proof}

\begin{Lem}\label{lem:normalQ}
Let $S$ be a domain containing $\bQ$, the field of rational numbers.
Then $S$ is a splinter if and only if $S$ is normal.
\end{Lem}
\begin{proof}
The ``only if'' part is contained in Lemma \ref{lem:SplinterNormal},
whereas the ``if'' part follows immediately from \cite[Lemmas 2 and 3]{Hoc73}.
\end{proof}

The following result is essentially due to Datta-Tucker \cite{DTetale}. 

\begin{Thm}\label{thm:SplinterEtale}
Let $S$ be a Noetherian splinter and let $S'$ be a direct limit of \'etale $S$-algebras.
Then $S'$ is a splinter.
\end{Thm}
\begin{proof}
Combine \cite[Theorem A]{DTetale} and Lemma \ref{lem:DlimSplinters}.
\end{proof}

In the rest of this section we establish a result parallel to \cite[Theorem 1.0.1]{DTopen}.
We first introduce the following definitions.

\begin{Def}[cf. {\cite[\S 4]{DTopen}}]
\label{def:LOCUS}
Let $X$ be a scheme.
The \emph{splinter locus} $\Spl(X)$ of $X$ is
the set of points $x\in X$ such that $\cO_{X,x}$ is a splinter.
For a ring $A$ we write $\Spl(A)$ for $\Spl(\Spec(A))$.
\end{Def}

Note that by Corollary \ref{cor:SplinterLocalize},
$\Spl(X)$ is stable under generalization and
a ring  $A$ is a splinter if and only if $\Spl(A)=\Spec(A)$.

\begin{Def}[cf. {\cite[Definition 3.2.1]{DTopen}}]\label{def:tandSgima}
Let $A$ be a ring and $B$ be an $A$-algebra.
The \emph{trace of $B$ over $A$} is the ideal
\[
\ft_{B/A}=\operatorname{im}(\operatorname{Hom}_A(B,A)\to A)
\]
where $\operatorname{Hom}_A(-,-)$ denotes the inner hom in the category of $A$-modules and the map is the evaluation at $1\in B$.

For a ring $A$ we let $\Sigma_A=\Sigma^{f,fp,s}_A$ be the collection of all ideals $\ft_{B/A}$ where $B$ is a finite and finitely presented $A$-algebra such that $\Spec(B)\to\Spec(A)$ is surjective.
\end{Def}

\begin{Prop}\label{prop:SplandSigma}
The followings hold for an arbitrary ring $A$.

(1) For all $\fp\in\Spec(A)$ we have $\Sigma_{A_\fp}=\{\fa A_\fp\mid \fa\in\Sigma_A\}$; and

(2) $\Spl(A)=\Spec(A)\setminus \bigcup_{\fa\in\Sigma_A}V(\fa)$.

In particular, if $\Sigma_A$ is finite, then $\Spl(A)$ is open.
\end{Prop}
\begin{proof}
(1) follows from Lemma \ref{lem:FFPSLocalize} and \cite[Lemma 3.2.3]{DTopen}.
For (2), we note that a ring map $A\to B$ splits as a map of $A$-modules if and only if $\ft_{B/A}=A$.
Thus (2) follows from (1).
\end{proof}

One more preparation.

\begin{Lem}\label{lem:extendFpure}
Let $R$ be a Noetherian local $\bF_p$-algebra. Assume that $R$ is $F$-pure.
Then there exists a flat local ring map $R\to R'$ such that $R'$ is Noetherian, local, $F$-pure, and $F$-finite.
\end{Lem}
\begin{proof}
Let $\fm$ be the maximal ideal of $R$, $k=R/\fm$, and let $k'$ be a field extension of $k$ that is $F$-finite (say the algebraic closure of $k$).
Let $R^\wedge$ be the completion of $R$, so $R^\wedge\cong k[[X_1,\ldots,X_n]]/I$ by Cohen structure theorem.
Let $R'= k'[[X_1,\ldots,X_n]]/Ik'[[X_1,\ldots,X_n]]$, a complete Noetherian local ring.
Then it is clear that $R\to R'$ is flat and local and that $R'$ is $F$-finite.
Finally, $R'/\fm R'=k'$, so $R'$ is $F$-pure by \cite[Lemma 3.26]{Has10} (take $c=1$).
\end{proof}

Here is our promised result.

\begin{Thm}\label{thm:OPENsmoverFpure}
Let $R$ be a Noetherian local $\bF_p$-algebra. Assume that $R$ is $F$-pure.
Then for every smooth $R$-scheme $X$,
$\Spl(X)$ is open.
\end{Thm}
\begin{proof}
We may assume $X=\Spec(A)$ affine.
Let $R\to R'$ be as in Lemma \ref{lem:extendFpure} and let $A'=A\otimes_R R'$.
Then $A'$ is a smooth $R'$-algebra, thus Noetherian, $F$-finite, and $F$-pure (see  \cite[Proposition 2.4]{Has10}), hence $F$-split.
By \cite[Proposition 3.4.1]{DTopen}, we see that $A'$ has finitely many uniformly $F$-compatible ideals, see \cite[Definition 3.1.1]{DTopen}.

For each $\fa\in\Sigma_A$ (Definition \ref{def:tandSgima}),
$\fa A'$ is a uniformly $F$-compatible ideal by \cite[Lemma 3.2.3]{DTopen}.
Since $A\to A'$ is faithfully flat, we see $\Sigma_A$ is finite, thus $\Spl(A)$ is open by Proposition \ref{prop:SplandSigma}.
\end{proof}

\noindent\emph{Remark}. We can weaken the assumption to that $X$ is covered by affine opens $\Spec(A)$ such that $R\to A$ is $F$-pure, see \cite[(2.1) and Proposition 2.4]{Has10}.

\section{Ultraproduct of local rings}\label{sec:Ultraproduct}

The notion of ultraproduct is ubiquitous in model theory.
For a general definition and treatment specializing to commutative algebra, see
\cite{BOOK}.
Here we give an equivalent and convenient definition for local rings.

\begin{Def}
Let $X$ be an index set and $A_x\ (x\in X)$ be a family of local rings.
An \emph{ultraproduct} of the rings $A_x\ (x\in X)$ is the localization of the ring $\prod_x A_x$ at a maximal ideal.

An \emph{ultrapower} of a local ring $A$ is an ultraproduct of a family $A_x\ (x\in X)$ where each $A_x=A$.
\end{Def}

\noindent\emph{Remark}.
The relation between our definition and the usual definition (cf. \cite[Chapter 2]{BOOK}) is as follows.
For a subset $U$ of $X$, we have an idempotent $e_U\in \prod_x A_x$ defined by $(e_U)_x=1$ when $x\in U$ and $(e_U)_x=0$ when $x\not\in U$.
Given a maximal ideal $\fM$ of $\prod_x A_x$, we get a family of subsets $\cU:=\{U\subseteq X\mid 1-e_U\in\fM\}$ of $X$.
Then $\cU$ is a (possibly principal) ultrafilter on $X$,
and the ultraproduct of the rings $A_x$ with respect to $\cU$
is the localization $(\prod_x A_x)_\fM$.
(This idea appears in the proof of Lemma \ref{lem:RegIntoUltrapower} below.)
Therefore our ultraproducts are ultraproducts in the usual sense.
Conversely, every ultrafilter $\cU$ comes from a maximal ideal $\fM$.
We leave the proof to the interested reader;
the fact that $A_x$ are local is essential.\\

Immediately from Lemma \ref{lem:SplinterProduct} and Corollary \ref{cor:SplinterLocalize}, we have

\begin{Lem}\label{lem:SplinterUltraprod}
An ultraproduct of local splinters is a splinter.
In particular, an ultrapower of a local splinter is a splinter.
\end{Lem}

For the next result, recall the following structure theorem of regular maps.
See \citestacks{07GC}.

\begin{Thm}[Popescu]\label{thm:Popescu}
A regular homomorphism of Noetherian rings $R\to A$ is a direct limit of smooth homomorphisms $R\to A_i$.
\end{Thm}

Note that an ultrapower of a local ring $A$ is canonically an $A$-algebra via the diagonal map.

\begin{Lem}[cf. {\cite[Theorem 7.1.1]{BOOK}}]
\label{lem:RegIntoUltrapower}
Let $R\to A$ be a regular homomorphism of Noetherian rings.
Assume that $R$ is local and strictly Henselian. 
Then there exists an ultrapower of $R$ that admits an $R$-algebra map from $A$.
\end{Lem}
\begin{proof}
Write $A=\mathrm{colim}_{i} A_i$ where each $A_i$ is smooth over $R$, possible by Popescu's Theorem \ref{thm:Popescu}.
For each $i,$ fix a section $\psi_i:A_i\to R$, possible as $R$ strictly Henselian (cf.
\citestacks{055U}).

Let $X$ be the family of all finitely generated $R$-subalgebras of $A$.
Each $x\in X$
is finitely presented over $R$ since $R$ is Noetherian, so the inclusion $x\subseteq A$ factors through some $A_i$, see \citestacks{00QO}.
Composing with $\psi_i,$
we get a section $\varphi^\circ_x:x\to R$.
Now, let $\varphi_x:A\to R$ be the map that is identically $0$ outside $x$ and $\varphi^\circ_x$ on $x$.
This is of course not a ring map in general, but we do get a map of sets
$\varphi: A\to R^X$ sending $a\in A$ to $(\varphi_x(a))_{x\in X}$.
We will find a maximal ideal $\fM$ of $R^X$ such that the composition $A\to R^X\to (R^X)_\fM$ is a map of rings.
Note that by construction, $\varphi$ restricts to the diagonal map on $R$, so this ring map will be an $R$-algebra map and the proof will be finished.

Now, for each $x\in X$,
let $e(x)$ be the idempotent of $R^X$ given by $e(x)_y=0$ when $x\subseteq y$ and $1$ otherwise.
For finitely many $x_1,\ldots,x_n\in X$ and an element $r$ in the ideal of $R^X$ generated by $e(x_1),\ldots,e(x_n)$, we always have $r_y=0$ if $y\in X$ contains all $x_1,\ldots,x_n$.
Such a $y$ always exists -- for example, it can be the $R$-subalgebra of $A$ generated by $x_1,\ldots,x_n$.
Therefore the elements $e(x)\ (x\in X)$ do not generate the unit ideal, so there exists a maximal ideal $\fM$ that contains all $e(x)\ (x\in X)$.
In particular $1-e(x)\not\in\fM$ for all $x\in X$.

We now show that the composition $A\to R^X\to (R^X)_\fM$ is a map of rings, that is, the elements
$\mu:=\varphi(a+b)-\varphi(a)-\varphi(b)$ and $\nu:=\varphi(ab)-\varphi(a)\varphi(b)$ are mapped to zero in $(R^X)_\fM$ for all $a,b\in A$. 
To this end, let $x_0$ be the $R$-subalgebra of $A$ generated by $a$ and $b$.
If $y\in X,x_0\subseteq y$, then $\varphi_y(c)=\varphi^\circ_y(c)$ for all $c\in x_0$.
Since $\varphi^\circ_y$ is a ring map, $\varphi_y(a+b)=\varphi_y(a)+\varphi_y(b)$ and $\varphi_y(ab)=\varphi_y(a)\varphi_y(b)$, so $\mu_y=\nu_y=0$.
Therefore $(1-e(x_0))\mu=(1-e(x_0))\nu=0$, as desired.
\end{proof}

\begin{Lem}
\label{lem:ultrapowerFlat}
Let $(R,\fm)\to (R',\fm')$ be a flat homomorphism of Noetherian local rings with $\fm'=\fm R'$.
Then any $R$-algebra map from $R'$ to an ultrapower $R_\natural$ of $R$ is pure.
\end{Lem}
\begin{proof}
Write $R_\natural=(R^X)_{\fM}$.
We know from \citetwostacks{05CZ}{05CY} that $R^X$ is flat over $R$.
We conclude that $R_\natural$ is flat over $R$.
Since $R'$ is also flat over $R$,
we have
\[
\operatorname{Tor}_1^{R'}(R_\natural,{R'}/\fm {R'})=\operatorname{Tor}_1^R(R_\natural,R/\fm )=0.
\]
Since $R'/\fm R'=R'/\fm'$ is a field, the $\fm R'$-adic completion $C$ of the $R'$-algebra $R_\natural$ is flat over $R'$ by
\citestacks{0AGW}.
Now note that $\fM$ lies above $\fm$ by \citestacks{0AME}.
Therefore $C/\fm C=R_\natural/\fm R_\natural$ (cf. \citestacks{05GG}) is nonzero and $C$ is faithfully flat over $R'$.
Therefore $R'\to C$ is pure and so is $R'\to R_\natural$.
\end{proof}

\section{Proof of main results}\label{sec:proofs}

First, a remark on regular ring maps.

\begin{Lem}\label{lem:DecRegTosh}
Let $A$ be a ring, $B$ a direct limit of \'etale $A$-algebras, $C$ a $B$-algebra.
Assume that $B$ and $C$ are Noetherian and that $A\to C$ is regular \citestacks{07BZ}.
Then $B\to C$ is regular.
\end{Lem}
\begin{proof}
Flatness of $B\to C$ follows from \citestacks{092C}, see \citetwostacks{092B}{092N}.
Now we need to show that for each $\fq\in\Spec(B)$ the Noetherian ring $C\otimes_B \kappa(\fq)$ is geometrically regular over $\kappa(\fq)$.
Let $\fp\in\Spec(A)$ be the preimage of $\fq$ in $A$.
The ring $C\otimes_{A}\kappa(\fp)$ is geometrically regular over $\kappa(\fp)$ and the ring $B\otimes_A \kappa(\fp)$ is a finite direct product of separable algebraic extensions of $\kappa(\fp)$, see \citestacks{0AH1}.
Therefore $C\otimes_B \kappa(\fq)$ is a direct factor of $C\otimes_{A}\kappa(\fp)$, and $\kappa(\fq)$ is separable algebraic over $\kappa(\fp)$.
We conclude by \citestacks{07QH}.
\end{proof}

\begin{proof}[Proof of Theorem \ref{thm:ScalarFoverR}]
There is a commutative diagram
\[
\begin{CD}
S@>>> S'\\
@VVV @VVV\\
S^{sh}@>>>S'^{sh}
\end{CD}
\]
where the vertival maps are strict Henselization, see \citestacks{04GU}.
The rings $S^{sh}$ and $S'^{sh}$ are Noetherian by \citestacks{06LJ}; and $S\to S'^{sh}$ is regular by \citetwostacks{07EP}{07QI};
thus $S^{sh}\to S'^{sh}$ is regular by Lemma \ref{lem:DecRegTosh}. 
By flat descent (Lemma \ref{lem:FlatDescentSplinters}) it suffices to show $S'^{sh}$ a splinter.
By Theorem \ref{thm:SplinterEtale},
we may therefore
assume $S$ strictly Henselian.

By Lemmas \ref{lem:RegIntoUltrapower} and \ref{lem:ultrapowerFlat},
there exists a pure ring map $S'\to S_\natural$, where $S_\natural$ is an ultrapower of $S$.
We conclude by Lemmas \ref{lem:FlatDescentSplinters} and \ref{lem:SplinterUltraprod}.
\end{proof}

\begin{proof}[Proof of Corollary \ref{cor:ScalarG}]
We need to show that our $S\to S'$ is regular.
Consider the commutative diagram
\[\begin{CD}
S @>>> S'\\
@VVV @VVV\\
S^\wedge @>>> S'^\wedge
\end{CD}\]
where the verical maps are completion.
The left vertical arrow is regular by the definition of a G-ring and the right vertical arrow is faithfully flat.
By \citetwostacks{07QI}{07NT}
it suffices to show $S^\wedge \to S'^\wedge$ regular.
This follows from \citestacks{07PM}, see \citestacks{0322}.
\end{proof}

\begin{proof}[Proof of Theorem \ref{thm:RegCharp}]
By Popescu's Theorem \ref{thm:Popescu} and Lemma \ref{lem:DlimSplinters}, we may assume $A$ a smooth $S$-algebra.
We must show $\Spl(A)=\Spec(A)$, see Definition \ref{def:LOCUS}.
By Theorem \ref{thm:SplinterEtale} and \citestacks{054L},
we may assume $A$ a polynomial ring over $S$,
and by induction on the number of variables,
it suffices to show $S[Y]$ a splinter, where $Y$ is an indeterminate.

By Corollary \ref{cor:SplinterLocalize},
we may assume $(S,\fm_S,k_S)$ local, and by induction on $\dim S$, that $S[Y]$ is a splinter outside
$Z:=V(\fm_S S[Y]).$
We may also assume $k_S$ separably closed, in particular infinite, by taking the strict Henselization of $S$.
Here we used Theorem \ref{thm:SplinterEtale} and Lemma \ref{lem:FlatDescentSplinters}, and the fact that the strict Henselization of a Noetherian local ring is a Noetherian local ring, see \citestacks{06LJ}.

Note that $S$ is a normal domain (Lemma \ref{lem:SplinterNormal}), hence so is $S[Y]$.
If $S$ contains $\bQ$,
then $S[Y]$ is a splinter by Lemma \ref{lem:normalQ}.
Thus we may assume $S$ is of prime characteristic.

Let $W=\Spl(S[Y])$.
Note that the splinter $S$ is $F$-pure; indeed, the Frobenius map $F:S\to F_*S$ is integral and induces a bijective map of spectra, so it is a direct limit of finite ring maps from $S$ that induces a surjective map of spectra, thus pure (cf. \citestacks{058J}).
Now $W$ is open in $\Spec(S[Y])=\bA^1_S$ by Theorem \ref{thm:OPENsmoverFpure}.
On the other hand, $W$ is stable under translations by all $a\in S$ (in fact all ring automorphisms of $S[Y]$); and
by Theorem \ref{thm:ScalarFoverR},
$S[Y]_{\fm_S S[Y]}$ is a splinter.
Thus $Z\cap W$ is a non-empty open subset of $Z$
stable under translations by all $a\in S$, thus corresponds to a non-empty open subset of $\bA^1_{k_S}$ stable under translations by all $a\in k_S$.
Such a set must be the whole of $\bA^1_{k_S}$ since $k_S$ is infinite.
Therefore $Z\cap W=Z$ and $W=\Spec(S[Y])$ as desired.
\end{proof}

\noindent\emph{Remark}.
The only reason the argument above cannot be applied in mixed characteristic
is that the openness result is not available.
The same proof as above yields Theorem \ref{thm:RegCharp} in arbitrary characteristic (possibly with an additional excellence hypothesis)
if one can prove a result similar to Theorem \ref{thm:OPENsmoverFpure} or \cite[Theorem 1.0.1]{DTopen} in mixed characteristic,
or even just that for every mixed-characteristic
Noetherian complete local splinter $S$, $\Spl(S[Y])$ is open, which in turn implies $\Spl(S[Y])=\Spec(S[Y])$.

However, we also suggest this to be not easy, since Theorem \ref{thm:RegCharp} in arbitrary characteristic implies the direct summand theorem.
See \cite[Remark 4.0.1]{DTetale}.


\begin{thebibliography}{Stacks}

\bibitem[AD21]{ADVal}
B. Antieau and R. Datta, \emph{Valuation rings are derived splinters}. Math. Zeitschrift 299 (2021), 827-851.

\bibitem[And74]{Andre}
M. Andr\'e, \emph{Localisation de la lissit\'e formelle}, Manuscripta Math. 13 (1974), 297-307.

\bibitem[And18]{And18}
Y. Andr\'e, \emph{La conjecture du facteur direct}, Publ. Math. Inst. Hautes Études Sci. 127 (2018), 71–93.

\bibitem[Bha18]{Bh18}
B. Bhatt,
\emph{On the direct summand conjecture and its derived variant}, Invent. Math. 212(2) (2018), 297–317.

\bibitem[DT19]{DTetale}
R. Datta and K. Tucker,
\emph{On some permanence properties of (derived) splinters}. (2019).
\url{https://arxiv.org/abs/1909.06891}

\bibitem[DT21]{DTopen}
R. Datta and K. Tucker,
\emph{Openness of splinter loci in prime characteristic}. (2021).
\url{https://arxiv.org/abs/2103.10525}

\bibitem[Has10]{Has10}
M. Hashimoto, \emph{$F$-pure homomorphisms, strong $F$-regularity, and $F$-injectivity.} Comm. Algebra 38.12 (2010), 4569–4596.

\bibitem[Hoc73]{Hoc73}
M. Hochster, \emph{Contracted ideals from integral extensions of regular rings}, Nagoya Math. J. 51 (1973), 25–43.

\bibitem[Kov17]{Kovacs}
S.J. Kov\'acs, \emph{Rational singularities}.
(2017).
\url{https://arxiv.org/abs/1703.02269}

\bibitem[Mat80]{Matsumura}
H. Matsumura, Commutative algebra. Second ed. Mathematics Lecture Note Series Vol. 56, Reading,
MA: Benjamin/Cummings Publishing Co., Inc. (1980).

\bibitem[Sch10]{BOOK}
H. Schoutens, The Use of Ultraproducts in Commutative Algebra. Lecture Notes in Mathematics 1999, Springer (2010).

\bibitem[Stacks]{Stacks}
The Stacks Project authors, \emph{The Stacks Project}.
(2022).
\url{https://stacks.math.columbia.edu/}
\end{thebibliography}
\end{document}